\documentclass[12pt,a4paper]{amsart}
\usepackage{amsmath,amssymb}
\usepackage{amsthm,graphicx, color}
\usepackage[T1]{fontenc}
\usepackage{float}

\newcommand{\NN}{\mathbb{N}}

\DeclareMathOperator{\inv}{inv}
\DeclareMathOperator{\fix}{fix}

\newtheorem{thm}{Theorem}
\newtheorem{lemma}[thm]{Lemma}

\newtheorem*{problem}{Problem}

\newcommand{\Sym}{\mathfrak{S}}
\newcommand{\Inv}{\mathfrak{I}}

\allowdisplaybreaks

\title[On $t$-extensions of the Hankel determinants]{On $t$-extensions of the Hankel determinants of certain automatic sequences}  
\date{February 6, 2014}
\author{Hao Fu}
\address{Institute for Interdisciplinary Information Sciences\\
Tsinghua University\\
Beijing,100084\\
P.R.China}
\email{fuhaots1992@gmail.com}

\keywords{Hankel determinant, $t$-Hankel determinant, permutation, involution,
Thue-Morse sequence,  period-doubling sequence,
regular paperfolding sequence}
\subjclass[2010]{05A05, 05A10, 05A15, 05A19, 11B50, 11B65, 11B85, 15A15.}

\author{Guo-Niu HAN}
\address{Institut de Recherche Math\'ematique Avanc\'ee\\
Universit\'e de Strasbourg et CNRS\\
7 rue Ren\'e-Descartes\\
 67084 Strasbourg\\
France}
\email{guoniu.han@unistra.fr}

\begin{document}
\begin{abstract}
In 1998,  Allouche, Peyri\`ere, Wen and Wen
considered the Thue--Morse sequence,
and proved that all the Hankel determinants of the period-doubling sequence 
are odd integral numbers. 
We speak of $t$-extension
when the entries along the diagonal in the Hankel determinant 
are all multiplied by~$t$.
Then we prove that the $t$-extension of each Hankel determinant of 
the period-doubling sequence 
is a polynomial in $t$, whose leading coefficient is
the {\it only one} to be an odd integral number.
Our proof makes use of the combinatorial set-up developed by Bugeaud and Han,
which appears to be very suitable for this study, as the parameter $t$ counts the number of fixed points of a permutation.
Finally, we prove that all the
$t$-extensions of the Hankel determinants of the regular paperfolding sequence are polynomials in $t$ of degree less than or equal to $3$. 
\end{abstract}
\maketitle

\section{Introduction} 
Hankel determinant is a very classical mathematical subject widely studied in 
Linear Algebra, Combinatorics, Number Theory and Algorithmics
(see, for example, \cite{Kr1998, Wimp2000, Layman2001, ERR2008, Ege2010}).
In particular,
the Hankel determinants of automatic sequences play an important role
in the study of irrationality exponents in Number Theory.
The first result on such determinants, 
obtained in 1998, 
is due to  Allouche, Peyri\`ere, Wen and Wen \cite{APWW1998},
who considered the Thue--Morse sequence,
and proved that all the Hankel determinants of the period-doubling sequence 
are odd integral numbers. 
This result allowed Bugeaud \cite{Bu2011} to prove that
the irrationality exponents of the Thue-Morse-Mahler numbers are exactly~2.

Let $x$ be a parameter. 
We identify each sequence ${\bf c}=(c_0, c_1, c_2, \ldots)$ with its generating function 
$C=C(x)=c_0+c_1x+c_2x^2+\cdots$. 
In general, the constant term $c_0$ will be equal to $1$.
For $k\ge 1$ and $p\ge 0$ let
\begin{equation}\label{def:Hankel}
H_{k}^{p}(C)=
H_{k}^{p}({\bf c}):=
\begin{vmatrix}
  c_p & c_{p+1}&\cdots & c_{p+k-1} \\
  c_{p+1} & c_{p+2}& \cdots & c_{p+k}\\
   \vdots & \vdots &\ddots & \vdots\\
   c_{p+k-1} & c_{p+k} & \cdots & c_{p+2k-2}
\end{vmatrix}
\end{equation}
be the $(p,k)$-order Hankel determinant of the series $C(x)$
or of the sequence ${\bf c}=(c_0, c_1, c_2, \ldots)$. 
We write $H_{k}(C):=H_{k}^{0}(C)$ for short.
The Thue--Morse sequence ${\bf e} = (1,-1,-1,1,\ldots)$ can be defined by the 
generating function
\begin{equation}\label{def:ThueMorse}
P_2(x)=\sum_{k=0}^\infty e_k x^k=\prod_{k=0}^\infty(1-x^{2^k}).
\end{equation}
Then, the {\it period-doubling sequence} ${\bf d}=(1,0,1,1,1,0,\ldots)$ 
is derived from the Thue--Morse sequence
by defining 
\begin{equation}\label{def:perioddoubling}
d_k=\frac{1}{2}|e_k - e_{k+1}| \qquad (k\geq 0). 
\end{equation}
\begin{thm}[APWW]
\label{thm:APWW}
For every positive integer $k$
the Hankel determinant $H_k({\bf d})$ of the period-doubling sequence ${\bf d}$ 
is an odd integral number.
In other words,
\begin{equation}
	H_k({\bf d}) \equiv 1 \pmod2.
\end{equation}
\end{thm}

Coons~\cite{Coons2013} considered the series 
\begin{equation}
\label{def:G00}
G_{0,0}(x):=\sum_{n=0}^{\infty}\frac{x^{2^{n}-1}}{1-x^{2^{n}}}
\end{equation}
and proved that 
all the Hankel determinants $H_k({G_{0,0}})$ of the power series $G_{0,0}(x)$  
are odd integral numbers. 
As shown in \cite{BH2013}, Coons's result
is essentially equivalent to Theorem \ref{thm:APWW}.
\medskip

Let $t$ be a parameter.
We speak of {\it $t$-extension}
when the entries along the diagonal in the $(p,k)$-order Hankel determinant 
are all multiplied by~$t$. 
In other words, we define the {\it $t$-Hankel determinant} of the formal power series $C(x)=c_0+c_1x+c_2x^2+\cdots$ by
\begin{equation}\label{def:tHankel}
H_{k}^{p}(C,t):=
\begin{vmatrix}
  c_pt & c_{p+1}&\cdots & c_{p+k-1} \\
  c_{p+1} & c_{p+2}t& \cdots & c_{p+k}\\
   \vdots & \vdots &\ddots & \vdots\\
   c_{p+k-1} & c_{p+k} & \cdots & c_{p+2k-2}t
\end{vmatrix}.
\end{equation}
Obviously, the above $t$-Hankel determinant (\ref{def:tHankel})
is a polynomial in $t$ of degree less than or equal to~$k$, 
which is equal to the traditional Hankel determinant (\ref{def:Hankel})
when $t=1$.
Again, we write $H_{k}(C,t):=H_{k}^{0}(C,t)$.
Our main result is stated as follows.
\begin{thm}
\label{thm:tAPWW}
For every positive integer $k$
the $t$-Hankel determinant $H_k({\bf d},t)$ of the period-doubling sequence ${\bf d}$ 
is a polynomial in $t$ of degree~$k$, whose leading coefficient is
the only one to be an odd integral number. In other words,
\begin{equation}
\label{equa:th}
H_{k}({\bf d},t)\equiv t^{k} \pmod2.
\end{equation}
\end{thm}

In the following table we reproduce the first values of the $t$-Hankel determinants 
of the period-doubling sequence ${\bf d}$. 
We see that all the coefficients are even integral numbers, except 
the coefficient of $t^k$.
When $t=1$ we recover Theorem \ref{thm:APWW}.

\bigskip

\centerline{
\begin{tabular}{|c|c|c|c|}
	  \hline
		$k$ & $H_k({\bf d},t)$ & $H_k({\bf d},t) \pmod2$ & $H_k({\bf d},1)$ \\
	 \hline
	 $0$ & $1$ & $1$ &$1$ \\
	  $1$ & $t$ &  $t$ &$1$ \\
	  $2$ & $t^2$ &  $t^2$ &$1$ \\
	  $3$ & $t^3 - 2t$ &  $t^3$ &$-1$ \\
	  $4$ & $t^4 - 4t^2$   &  $t^4$ &$-3$ \\
	  $5$ & $t^5 - 6t^3 + 2t^2 + 4t$   &  $t^5$ &$1$ \\
	  $6$ & $t^6 - 8t^4 + 4t^3 + 12t^2 - 8t$  &  $t^6$ &$1$ \\
	  $7$ & $t^7 - 12t^5 + 10t^4 + 24t^3 - 24t^2$ &  $t^7$ &$-1$ \\
	  $8$ & $t^8 - 16t^6 + 16t^5 + 48t^4 - 64t^3$ &  $t^8$ &$-15$\\
	  \hline
\end{tabular}
}
\bigskip
Actually, Theorem \ref{thm:APWW} has three proofs. The first one is due to 
Allouche, Peyri\`ere, Wen and Wen by using
determinant manipulation \cite{APWW1998},
which consists of proving sixteen recurrence relations between determinants.
The second one is a combinatorial proof derived by Bugeaud and Han \cite{BH2013}.
The third proof is very short by using the Jacobi continued fraction algebra \cite{Han2013}.
For proving Theorem \ref{thm:tAPWW} it seems 
the method used in the second proof is more suitable, as
the parameter $t$ counts the number of fixed points of permutations.

Some basic notations and properties on permutations and involutions
are collected in Section \ref{section:permutation}, 
including the statement of the key combinatorial result, namely, Theorem \ref{thm:key}. 
The proof of the main result
(Theorem \ref{thm:tAPWW}) is found in
Section \ref{section:tAPWW}, after proving Theorem \ref{thm:key}.

\medskip

The {\it regular paperfolding sequence} ${\bf r}=(1,1,0,1,1,0,0,\ldots)$
can be defined by the generating function \cite{WikipediaRP, Allouche1987}
\begin{equation}\label{def:G02}
G_{0,2}(x)=\sum_{n\geq 0} r_{n} x^n = \sum_{n=0}^\infty 
\frac{x^{2^n-1}}{1-x^{2^{n+2}}}.
\end{equation}
Coons and Vrbik conjectured \cite{CoonsVrbik2012} and Guo, Wu and Wen \cite{GWW2013} 
proved the following result.
\begin{thm}[GWW]
	\label{thm:GWW}
The Hankel determinants of the regular paperfolding sequence ${\bf r}$ are periodic of period $10$. More precisely, we have
\begin{equation}\label{eq:period10}
(H_k({\bf r}))_{k=0,1,\ldots} \equiv (1,1,1,0,0,1,0,0,1,1)^* \pmod 2.
\end{equation}
\end{thm}
An automatic proof of Theorem \ref{thm:GWW} by a computer algebra system is described in \cite{Han2014}.
Our second result is stated next.
\begin{thm}
\label{thm:tPaperFolding}
For every positive integer $k$
the $t$-Hankel determinant $H_k({\bf r},t)$ of the 
regular paperfolding sequence ${\bf r}$ 
is a polynomial in $t$ of degree less than or equal to $3$.
\end{thm}

Theorem \ref{thm:tPaperFolding} is proved in 
Section~\ref{section:tPaperFolding}. 
In the following table we reproduce the first values of the $t$-Hankel determinants of the regular paperfolding sequence ${\bf r}$. 
We see that all the $H_k({\bf r}, t)$'s are polynomials of degree less than or equal to $3$.

\bigskip
\centerline{
\begin{tabular}{|c|c||c|c|}
	  \hline
		$k$ & $H_k({\bf r},t)$   &  $k$ & $H_k({\bf r},t)$    \\
	 \hline
	 $0$ & $1$                 & $5$ & $-t^3 + 2t^2 + 2t - 2$      \\
	 $1$ & $t$                &  $6$ & $2t^2 - 2t - 4$             \\
	 $2$ & $-1$              &   $7$ & $3t^3 - 6t^2 - 7t + 6$      \\
	 $3$ & $-2t$            &    $8$ & $-9t^2 + 12t + 16$          \\
	 $4$ & $-t^2 + 2t + 1$ &     $9$ & $-15t^3 + 20t^2 + 46t - 40$   \\
	  \hline
\end{tabular}
}
\medskip

As earlier mentioned, Theorem \ref{thm:tAPWW} is a $t$-extension of Theorem \ref{thm:APWW}.
However, Theorem \ref{thm:GWW} can not be obtained from Theorem \ref{thm:tPaperFolding} by
specializing $t=1$.  The following problem remains unsolved.

\begin{problem}
Find a true $t$-extension of Theorem \ref{thm:GWW}. In other words,
find a property of the $t$-Hankel determinants of
the regular paperfolding sequence, which implies relation (\ref{eq:period10}) when $t=1$. 
\end{problem}

\medskip

\section{Permutations and involutions}\label{section:permutation} 
A combinatorial set-up, based on permutations and involutions, for studying 
the Hankel determinants of the period-doubling sequence was 
introduced in \cite{BH2013}. We propose a refinement of such a
combinatorial set-up for studying $t$-Hankel determinants.
The following infinite sets of integers play an important role.
\begin{align*}
N&=\NN^0=\{0,1,2,3,\ldots\},\\
J&=\{(2n+1)2^{2k}-1\mid n,k\in N\}=\{0,2,3,4,6,8,10,11,12,14,\cdots\},\\
J^{\ast}&=\{(2n+1)2^{2k}-1\mid n,k\in N,k>0\}=\{3,11,15,19,27,35,\cdots\},\\
K&=N\setminus J=\{(2n+1)2^{2k+1}-1\mid n,k\in N\}=\{1,5,7,9,13,17,\cdots\},\\
L&=N\setminus J^{\ast}=K\cup \{2n\mid n\in N\}=\{0,1,2,4,6,8,10,13,14,16,\cdots\},\\
P&=\{k\mid k\equiv 0,3 \pmod 4\}=\{0,3,4,7,8,11,12,15,16,\cdots\},\\
Q&=\{k\mid k\equiv 1,2 \pmod 4\}=\{1,2,5,6,9,10,13,14,17,\cdots\}.
\end{align*}
For each infinite set $A$ let $A|_{m}$ be the finite set composed of the 
smallest $m$ integers in $A$.

Let $\Sym_m=\Sym_{\{0,1,\cdots,m-1\}}$ be the set of all permutations 
on $N|_m$. A permutation is represented by the product of its disjoint cycles. 
For example, the permutation $\sigma =(0,5)(1)(2,6,3)(4,8)(7)$ is
an element from $\Sym_{9}$.
An $involution$ is a permutaion $\sigma$ such that $\sigma=\sigma^{-1}$. 
Equivalently, a permutation $\sigma$ is an involution if each cycle of $\sigma$
is either a fixed point $(b)$ or a {\it transposition} $(c,d)$. For instance,
$\sigma =(0,5)(1)(2,6)(3)(4,8)(7)\in \Sym_{9}$ is an involution.
For each set $B$, a transposition $(c,d)$ is said ``{\it in $B$}''  
if $c+d\in B$. In this case, we write $(c,d)\in B$.

For a nonnegative integer~$k$ and two sets of positive integers $A,B$ 
such that $A$ is finite,
let $\mu(A,k,B)$ be the number of involutions $\sigma$ in $\Sym_A$ 
having exactly $k$ transpositions
such that all transpositions of $\sigma$ are in $B$.
The following
key result is useful for proving Theorem \ref{thm:tAPWW} (see Section \ref{section:tAPWW}). 
\begin{thm}
\label{thm:key}
For $m\ge 1$ and $k\ge 0,$ we have 
\begin{equation}
\mu(N|_{m},k,J)\equiv\begin{cases}
	1 \pmod2, & \text{if $k=0$}; \\
	0 \pmod2, & \text{if $k\geq 1$}. \\
\end{cases}
\end{equation}
\end{thm}
The proof of Theorem \ref{thm:key} is given  
in Section \ref{section:tAPWW}, with the help of several lemmas
stated in the remainder of this section.
\begin{lemma}
\label{lemma:NPQ}
For $m\ge 1$ and $k\geq 0$ we have 
\begin{equation}
\label{equ:np}
\mu(N|_{m},k,J)=\mu(P|_{m},k,L)
\end{equation}
and 
\begin{equation}
\label{equ:pq}
\mu(P|_{m},k,J^{\ast})=\mu(Q|_{m},k,J^{\ast}).
\end{equation}
\end{lemma}
\begin{proof}
We define two transformations:
\begin{align*}
\beta:N\rightarrow P;& \quad \quad \ell\mapsto \begin{cases}
2\ell, & \quad\text{if $\ell$ is even}; \\
2\ell+1, & \quad\text{if is $\ell$ is odd};
\end{cases}\\
\delta:P\rightarrow Q;& \quad \quad \ell\mapsto \begin{cases}
\ell+1, & \quad\text{if $\ell$ is even}; \\
\ell-1, & \quad\text{if $\ell$ is odd}.
\end{cases}
\end{align*}
The transformation $\beta$ is a  bijection of $N|_{m}$ onto  $P|_{m}$,
and can be extended to the set of all involutions on $N|_m$ 
by applying $\beta$ on every letter of the involutions.
For example
$$
\beta((7)(0,5),(6,3),(1),(8,2),(4))=(15)(0,11)(12,7)(3)(16,4)(8).
$$
We now claim that, for any $c,d\in N|_{m}$, the transposition 
$(c,d)$ is in $J$ if and only if  $(\beta(c),\beta(d))$ is in $L$.
The proof of this claim works by distinguishing the parities of $c$ and $d$:
(i) if $c$ and $d$ are even, then $\beta(c)=2c$ and $\beta(d)=2d$, 
so that $\beta(c)+\beta(d)$ is even and is in $L$;
(ii) if $c$ and $d$ are odd, then $\beta(c)=2c+1$ and $\beta(d)=2d+1$, 
so that $\beta(c)+\beta(d)$ is even  and is in $L$;
(iii) if $c+d\in J$ and one of the intergers $c,d$ is even, the other being odd.
Then,
$$
\beta(c)+\beta(d)=2c+2d+1=2\times((2n+1)2^{2k}-1)+1=(2n+1)2^{2k+1}-1 \in L.
$$
The ``reverse part'' is proved in the same manner.
Thus, equation (\ref{equ:np}) holds.

The transformation $\delta$ is a  bijection of $P|_{m}$ onto  $Q|_{m}$,
and can be extended to the set of all involutions on $P|_m$ 
by applying $\delta$ on every letter of the involutions.
For example
$$
\delta((15)(0,11)(12,7)(4)(16,3)(8))=(14)(1,10)(13,6)(5)(17,2)(9).
$$
If the transposition $(c,d)$ is in $J^{\ast}$ and $c,d\in P$, 
then one of the integers $c,d$ is even, the other being odd.
Hence,
$$
\delta(c)+\delta(d)=c-1+d+1=c+d\in J^{\ast}.
$$
Thus, equation (\ref{equ:pq}) is proved.
\end{proof}
\begin{lemma}
\label{lemma:2n_n}
For each $k\geq 0$ we have
\begin{equation}
\mu(N|_{2n},k,J^{\ast})\equiv\begin{cases}
	0 \pmod2, & \text{if $k$ is odd} ;\\
	\mu(P|_{n},k/2,J^{\ast}) \pmod2,& \text{if $k$ is even}.
\end{cases}
\end{equation}
\end{lemma}
\begin{proof}
It is easy to see that, if $c+d\in J^{\ast}$, then $c+d\equiv 3$ (mod 4). 
Thus, both $c$ and $d$ belong either to $P$ or to $Q$.
Hence,
\begin{eqnarray}
\mu(N|_{2n},k,J^{\ast})
&=&\sum_{i+j=k}\mu(P|_{n},i,J^{\ast})\ \mu(Q|_{n},j,J^{\ast}) \\
&=&\sum_{i+j=k}\mu(P|_{n},i,J^{\ast})\ \mu(P|_{n},j,J^{\ast}). \label{eq:muPP}
\end{eqnarray}
The last identity holds by Lemma (\ref{lemma:NPQ}).
When $k=2\ell+1$ is odd, the right-hand side of equation (\ref{eq:muPP})
is equal to
\begin{equation*}
2\sum_{i=0}^{\ell}\mu(P|_{n},i,J^{\ast})\ \mu(P|_{n},2\ell+1-i,J^{\ast})\equiv 0 \pmod2.
\end{equation*}
When $k=2\ell$ is even, we have 
\begin{eqnarray*}
	&&\sum_{i+j=2\ell}\mu(P|_{n},i,J^{\ast})\ \mu(P|_{n},j,J^{\ast})\\
 &=& 2\sum_{i=0}^{\ell-1}\mu(P|_{n},i,J^{\ast})\ \mu(P|_{n},2\ell-i,J^{\ast})
	+\mu(P|_{n},\ell,J^{\ast})\ \mu(P|_{n},\ell,J^{\ast})\\
 &\equiv& \mu(P|_{n},k/2,J^{\ast}) \pmod2.
\end{eqnarray*}
This achieves the proof.
\end{proof}

In the sequel, the notation $a\equiv b$ means that the integers $a$ and $b$ are congruent modulo 2 when nothing else is specified. 
\begin{lemma}\label{lemma:sumP}
For $m\ge 1$ and $k\ge 1$ we have
\begin{equation}
\label{equ:sumP}
\sum_{i=0}^{k}\mu(P|_{m},i,J^{\ast})\binom{m-2i }{ 2k-2i}
\equiv \mu(P|_{m},k,L) \pmod2.
\end{equation}
\end{lemma}
\begin{proof}
Recall that $\mu(A,k,B)$ is the number of involutions $\sigma$ in $\Sym_A$ 
having exactly $k$ transpositions 
such 
that all transposions of $\sigma$  are in $B$.
For two disjoint sets of integers $B_1$ and $B_2$,
we define $\mu(A,k_1,k_2,B_1, B_2)$ to be the number of involutions $\sigma$
in $\Sym_A$ 
having exactly $k_1$ transpositions in $B_1$ and $k_2$ transpositions in $B_2$
such that all transposions are in $B_1\cup B_2$. 
So that $\mu(A,0,k_2,B_1, B_2)=\mu(A,k_2,B_2)$.
\medskip

Let $i$ and $j$ be two positive integers such that $0\leq i\leq j\leq k$. 
Consider the set $\Inv_{j}$ of involutions $\sigma$ on $P|_m$ having 
have exactly $j$ transpositions in $J^\ast$ and $k-j$ transpositions in $L$
and no other transposition.
Then, the cardinality of $\Inv_j$ is equal to
$\mu(P|_{m},j,k-j,J^{\ast},L)$. 
A marked involution is obtained
from an involution $\sigma\in\Inv_j$ by coloring $i$ transpositions 
among the $j$ transpositoins in $J^\ast$. Let $\Inv_{i,j}$ be
the set of all those marked involutions. 
The cardinality of $\Inv_{i,j}$ is equal to
$\binom{j}{ i}\ \mu(P|_{m},j,k-j,J^{\ast},L)$. 
Hence, the cardinality of the set 
$\Inv_{i,\bullet}= \Inv_{i,i}+\Inv_{i,i+1}+\cdots+\Inv_{i,k}$ is equal to
\begin{equation}\label{equ:card1}
	\sum_{j=i}^k \binom{j}{ i}\ \mu(P|_{m},j,k-j,J^{\ast},L).
\end{equation}
On the other hand, the marked involutions in $\Inv_{i, \bullet}$ can be
enumerated as follows. 
Consider the involutions on $P|_m$ that have exactly $i$ transpositions in
$J^{\ast}$, which are said to be colored.  
There are $\mu(P|_{m},i,J^{\ast})$ such involutions. 
Then ramdomly choose $2k-2i$ letters from the rest $m-2i$ original 
fixed points on $P|_m$, to generate another $k-i$ transpositions, which are either in $J^{\ast}$ or in $L$. We get a marked involution which has 
exactly $i+(k-i)=k$ transpositions. 
Hence, the cardinality of the set 
$\Inv_{i,\bullet}$ is equal to
\begin{equation}
\label{equ:ast}
\mu(P|_{m},i,J^{\ast})\binom{m-2i }{ 2k-2i}(2k-2i-1)(2k-2i-3)\cdots 3\cdot 1.
\end{equation}
Hence, the two quantities  (\ref{equ:card1}) and (\ref{equ:ast})
are equal. We have successively
\begin{align*}
	&\sum_{i=0}^{k}\mu(P|_{m},i,J^{\ast})\binom{m-2i }{ 2k-2i}\\
	\equiv&\sum_{i=0}^{k}\mu(P|_{m},i,J^{\ast})[\binom{m-2i }{ 2k-2i}(2k-2i-1)(2k-2i-3)\cdots(3)(1)]\\
	=&\sum_{i=0}^{k}\sum_{j=i}^{k}\binom{j }{ i}\mu(P|_{m},j,k-j,J^{\ast},L)\\
	=&\sum_{j=0}^{k}(\sum_{i=0}^{j}\binom{j }{ i})\mu(P|_{m},j,k-j,J^{\ast},L)\\
=&\sum_{j=0}^{k}2^{j}\mu(P|_{m},j,k-j,J^{\ast},L)\\
\equiv& \mu(P|_{m},0,k,J^{\ast},L)\\
=&\mu(P|_{m},k,L).
\end{align*}
This achieves the proof.
\end{proof}

\section{Proofs of Theorems \ref{thm:key} and \ref{thm:tAPWW}}\label{section:tAPWW} 
Firstly, we establish two lemmas about congruences for binomial coefficients.
\begin{lemma}
\label{lemma:bino1}
For $n, k \geq 0$ we have
\begin{equation}
	\sum_{i+j=k}\binom{n }{ 2i}\binom{n }{ 2j}\equiv\begin{cases}
	0 \pmod2, &   \text{if $k$ is odd}; \\
		\binom{n}{ k} \pmod2, & \text{if $k$ is even}.
\end{cases}
\end{equation}
\end{lemma}
\begin{proof}
If $k=2\ell+1$ is odd, then 
\begin{equation*}
	\sum_{i+j=2\ell+1}\binom{n }{ 2i}\binom{n }{ 2j}
	=2\sum_{i=0}^{\ell}\binom{n }{ 2i}\binom{n }{ 4\ell+2-2i}
\equiv 0\quad\text{(mod 2)}.
\end{equation*}
If $k=2\ell$ is even,  then
\begin{eqnarray*}
	\sum_{i+j=2\ell}\binom{n }{ 2i}\binom{n }{ 2j}
	&=&2\sum_{i=0}^{\ell-1}\binom{n }{ 2i}\binom{n }{ 4\ell-2i}
	+\binom{n }{ 2\ell}\binom{n }{ 2\ell}\\
	&\equiv& \binom{n}{ k} \pmod2.
\end{eqnarray*}
This achieves the proof.
\end{proof}
\begin{lemma}
\label{lemma:bino2}
For $n,m,k\geq 0$ such that $n+m$ is odd we have
\begin{equation}\label{equ:bino2}
	\sum_{i+j=k}\binom{n }{ 2i}\binom{m }{ 2j}\equiv \binom{n+m }{ 2k} \pmod2.
\end{equation}
\end{lemma}
\begin{proof} 
We have
\begin{eqnarray*}
	\binom{n+m}{ 2k}
	&=&\sum_{i+j=2k} \binom{n}{ i}\binom{m}{ j} \\
 &=&\sum_{i+j=k} \binom{n}{ 2i}\binom{m}{ 2j}
	+\sum_{i+j=k-1} \binom{n}{ 2i+1}\binom{m}{ 2j+1}.
\end{eqnarray*}
Since $\binom{2a }{ 2b+1}$  is even for any positive integers $a$ and $b$ \cite{Lucas}, 
\begin{equation}
	\binom{n}{ 2i+1}\binom{m}{ 2j+1}\equiv 0 \pmod2,
\end{equation}
if $n$ or $m$ is even.  This is true because $n+m$ is odd.
Equation (\ref{equ:bino2}) holds.
\end{proof}
Secondly, we prove Theorem \ref{thm:key} by induction.
\begin{proof}[Proof of Theorem \ref{thm:key}]
When $k=0$, the quantity $\mu(N|_{m},k,J)$ counts the involutions $\sigma$ 
without any
transposition. It means that every letter of $\sigma$ is a fixed point, 
so that $\mu(N|_{m},0,J)=1$.

\smallskip
When $k\geq 1$, two cases are to be considered.
Notice that 
any transposition of type $(even,even)$ or $(odd,odd)$ is in $J$ since
$J$ contains all even integers.
Let $k_1+k_2=k$. An involution $\sigma$ having exactly $k$ transpositions 
in $J$ can be generated from an involution $\tau$ having exactly $k_1$
transpositions in $J^\ast$ by adding $k_2$ transpositoins 
in $J\setminus J^\ast$. The latter $k_2$ transpositions are of type
$(even,even)$ or $(odd,odd)$, and are easy to count by using binomial coefficients.

\smallskip
\goodbreak
(i) When $m=2n$ is even and $k\geq 1$ we have
\begin{eqnarray*}
&& \mu(N|_{2n},k,J)\\
&=& \sum_{k_1+k_2=k}\mu(N|_{2n},k_1,J^{\ast})
	\sum_{i+j=k_2}\left[\binom{n-k_1 }{ 2i}(2i-1)(2i-3)\cdots 1 \right.\\
	&&\kern 2cm	\left.\times \binom{n-k_1 }{ 2j}(2j-1)(2j-3)\cdots 1\right]\\
 &\equiv &\sum_{k_1+k_2=k}\mu(N|_{2n},k_1,J^{\ast})\sum_{i+j=k_2}\binom{n-k_1 }{ 2i}\binom{n-k_1}{ 2j} \pmod2.
\end{eqnarray*}
If $k$ is odd, then one of the $k_1,k_2$ is odd and the other is even. 
By Lemma \ref{lemma:bino1} and 
Lemma \ref{lemma:2n_n}, $\mu(N|_{2n},k,J)\equiv 0\pmod2$.
If $k=2\ell$ is even, then
\begin{eqnarray*}
&&\mu(N|_{2n},k,J)\\
&=&\sum_{k_1+k_2=2l}\mu(N|_{2n},k_1,J^{\ast})\sum_{i+j=k_2}\binom{n-k_1 }{ 2i}
	\binom{n-k_1 }{ 2j}\\
	&\equiv &\sum_{k_1+k_2=\ell}\mu(N|_{2n},2k_1,J^{\ast})\!\!\sum_{i+j=2k_2}\!\binom{n-2k_1 }{2i}\binom{n-2k_1 }{ 2j}  \text{[By Lemma \ref{lemma:bino1}]}\\
 &\equiv &\sum_{k_1+k_2=\ell}\mu(P|_{n},k_1,J^{\ast})\binom{n-2k_1 }{ 2k_2}
	\text{\qquad\qquad\quad [By Lemmas \ref{lemma:2n_n} and \ref{lemma:bino1}]}\\
	&\equiv &\mu(P|_{n},\ell,L) \text{\qquad\qquad [By Lemma \ref{lemma:sumP}]}\\
 &=&\mu(N|_{n},k/2,J) \text{\ \ \qquad [By Lemma \ref{lemma:NPQ}]}\\
		&\equiv& 0\pmod2. \text{\ \quad\qquad [By induction]}
\end{eqnarray*}

\smallskip
(ii) When $m=2n+1$ is odd and $k\geq 1$, we successivly have
\begin{eqnarray*}
&& \mu(N|_{2n+1},k,J)\\
&=&\!\sum_{k_1+k_2=k}\!\mu(N|_{2n+1},k_1,J^{\ast})
	\!\!\sum_{i+j=k_2}\!\!\left[\binom{n+1-k_1 }{ 2i}(2i-1)(2i-3)\cdots 1\right.\\
  	&&\qquad\qquad\left.\times \binom{n-k_1 }{ 2j}(2j-1)(2j-3)\cdots 1\right]\\
&\equiv &\sum_{k_1+k_2=k}\mu(N|_{2n+1},k_1,J^{\ast})\sum_{i+j=k_2}\binom{n+1-k_1 }{ 2i}\binom{n-k_1}{ 2j}\\
	&\equiv &\sum_{k_1+k_2=k}
		\left[\sum_{r+s=k_1}\mu(P|_{n+1},r,J^{\ast})\mu(Q|_{n},s,J^{\ast})\right]
		\binom{2n+1-2k_1 }{ 2k_2}, \\
\end{eqnarray*}
where the last identity is obtained by using Lemma \ref{lemma:bino2}. 
Applying Lemmas \ref{lemma:NPQ} and \ref{lemma:bino2} to the above quantity we get
\goodbreak
\begin{eqnarray*}
&& \mu(N|_{2n+1},k,J)\\
	&\equiv &\sum_{k_1+k_2=k}\left[\sum_{r+s=k_1}\mu(P|_{n+1},r,J^{\ast})\mu(P|_{n},s,J^{\ast})\right]\\
	&& \qquad\qquad \times 
		\sum_{i+j=k_2}\binom{n+1-2r}{ 2i}\binom{n-2s }{ 2j} \\
	&=&\sum_{r+s+i+j=k}\mu(P|_{n+1},r,J^{\ast})\binom{n+1-2r}{ 2i}\mu(P|_{n},s,J^{\ast})\binom{n-2s }{ 2j}\\
	&=&\sum_{k_1+k_2=k}\left[\sum_{r+i=k_1}\mu(P|_{n+1},r,J^{\ast})\binom{n+1-2r}{ 2i}\right.\\
	&& \qquad \qquad \times \left.\sum_{s+j=k_2}\mu(P|_{n},s,J^{\ast})\binom{n-2s }{ 2j}\right]\\
 &\equiv&\sum_{k_1+k_2=k} \mu(P|_{n+1},k_1,L)\mu(P|_{n},k_2,L)  \text{\qquad\quad [By Lemma \ref{lemma:sumP}]}\\ 
 &\equiv&\sum_{k_1+k_2=k} \mu(N|_{n+1},k_1,J)\mu(N|_{n},k_2,J). \text{\qquad\quad[By Lemma \ref{lemma:NPQ}]} \\
		&\equiv& 0 \pmod 2. \text{\qquad\quad[By induction]}
\end{eqnarray*}
This achieves the proof.
\end{proof}
Lastly, Theorem \ref{thm:tAPWW} on the $t$-extensions of the Hankel determinants
of the period-doubling sequence is proved as follows.
Keep in mind the infinite set
$$
J=\{(2n+1)2^{2k}-1|n,k\in N\}=\{0,2,3,4,6,8,10,11,12,14,\cdots\},
$$
and the period-doubling sequence defined by (\ref{def:perioddoubling}).
In \cite{BH2013} Bugeaud and Han proved the following result.
\begin{lemma}
\label{lemma:J}
For $k\geq 0$, 
the integer $d_k$ is odd if, and only if, $k$ is in $J$.
\end{lemma}

\begin{proof}[Proof of Theorem \ref{thm:tAPWW}]
Let $D(x)$ be the generating function of the period\-doubling sequence
$$
D(x)=\sum_{k\ge 0}d_{k}x^{k}=1+x^2+x^3+x^4+x^6+\cdots
$$
Let $k$ be a positive integer. By Leibniz formula for 
determinants \cite{Determinant}, 
the $t$-Hankel determinant $H_{k}(D,t)$ is equal to 
\begin{equation}\label{equ:tdet}
	\sum_{\sigma\in \Sym_{k}} t^{\fix(\sigma)}(-1)^{\inv(\sigma)}j_{0+\sigma(0)}j_{1+\sigma(1)}\cdots j_{k-1+\sigma(k-1)},
\end{equation}
where $\inv(\sigma)$  is the number of inversions of $\sigma$
and $\fix(\sigma)$  is the number of fixed points of $\sigma$ defined by
\begin{eqnarray*}
	\inv(\sigma) &=& \#\{(i,j) \mid 0\leq i<j\leq k-1, \sigma(i)>\sigma(j) \}; \\
	\fix(\sigma) &=& \#\{i \mid 0\leq i\leq k-1, \sigma(i)=i \}.
\end{eqnarray*}
The product 
$$j_{0+\sigma(0)}j_{1+\sigma(1)}\cdots j_{k-1+\sigma(k-1)}$$ 
is equal to 1 if $i+\sigma(i)\in J$ for $i=0,1,\cdots, k-1$, 
and is equal to 0 otherwise. 
Let $\sigma$ be a permutation such that $\sigma\not=\sigma^{-1}$.
We have $\inv(\sigma)=\inv(\sigma^{-1})$ and 
$\fix(\sigma)=\fix(\sigma^{-1})$. Accordingly, they have the same contribution
to summation (\ref{equ:tdet}), and can be deleted. Hence
\begin{equation}\label{equ:tdetInvolution}
H_{k}(D,t)\equiv
	\sum_{\sigma} t^{\fix(\sigma)}j_{0+\sigma(0)}j_{1+\sigma(1)}\cdots j_{k-1+\sigma(k-1)} \pmod2,
\end{equation}
where the sum is over the set of all involutions $\sigma$ on $N|_k$.
By Theorem~\ref{thm:key}, 
$$
H_{k}(D,t)\equiv \sum_{i=0}^{k}t^{k-2i}\mu(N|_{k},i,J)\equiv t^{k}\mu(N|_{k},0,J)=t^{k}.
$$
This achieves the proof.
\end{proof}


\section{Regular paperfolding sequence}\label{section:tPaperFolding}  
We define the infinite set
\[
R=\{(4k+1)2^{n}-1\mid n,k\in N\}=\{0,1,3,4,7,8,9,12,15,16,\cdots\}.\\
\]
Notice that, for each integer $m$ in the set $R$, there are unique 
integers $n$ and $k$ such that $(4k+1)2^n-1=m$.
Recall the regular paperfolding sequence ${\bf r}=(r_k)_{k=0,1,2,\ldots}$ defined by~(\ref{def:G02}). 
We have the following result.
\begin{lemma}
\label{lemma:S}
For each $k\ge 0$ the integer $r_{k}$ is equal to 1 if and only if $k$ is in $R$, and is equal to 0 otherwise.
\end{lemma}
\begin{proof}
By definition of (\ref{def:G02}), we have 
\begin{align*}
G_{0,2}(x)=&\sum_{n\geq 0} r_{n} x^n = \sum_{n=0}^\infty \frac{x^{2^n-1}}{1-x^{2^{n+2}}}\\
	=&\sum_{n=0}^{\infty}x^{2^n-1}\left(\sum_{k\ge 0}(x^{2^{n+2}})^k\right)\\
=&\sum_{n,k\ge 0}x^{4k\cdot 2^{n}+2^{n}-1}.
\end{align*}
Thus the lemma holds.
\end{proof}

\begin{proof}[Proof of Theorem \ref{thm:tPaperFolding}]
As discussed in Section \ref{section:tAPWW}, the $t$-Hankel 
determinant $H_{k}({\bf r},t)$ is equal to 
\begin{equation}\label{equ:tret}
	\sum_{\sigma\in \Sym_{k}} t^{\fix(\sigma)}(-1)^{\inv(\sigma)}r_{0+\sigma(0)}r_{1+\sigma(1)}\cdots r_{k-1+\sigma(k-1)}.
\end{equation}
The product 
$$r_{0+\sigma(0)}r_{1+\sigma(1)}\cdots r_{k-1+\sigma(k-1)}$$ 
is equal to 1 if $i+\sigma(i)\in R$ for $i=0,1,\cdots, k-1$, 
and is equal to 0 otherwise.

Recall the three representations for permutations: 
the {\it one-line}, {\it two-line} and {\it product of disjoint cycles}. 
For example, we write
$$
\sigma\in\Sym_9=516280374=
\begin{pmatrix}
012345678\\ 516280374\\
\end{pmatrix}
= (0,5)(1)(2,6,3)(4,8)(7). 
$$
Consider a permutation $\sigma$ having
at least 4 fixed points, i.e., $\fix(\sigma)\ge 4$. 
It's easy to know that an even number $m$ is in $R$ if and only if $m\equiv 0\pmod 4$, so that all fixed points are even.
Consequently, there are at 
least 3 columms of type $\binom{odd}{odd}$ in the two-line representation of the permutation $\sigma$. 
Let 
$\binom{i_1}{j_1}$, $\binom{i_2}{ j_2}$ and $\binom{i_3}{ j_3}$ 
be the first three such columns.
By the Pigeonhole Principle, there are at least two numbers among $j_1,j_2,j_3$ which are congruent modulo 4. 
Without loss of generality, we assume that $j_1$ and $j_2$ are congruent modulo 4. (When all three numbers are congruent, we also choose $j_1$ and $j_2$). 
We define another permutation $\tau$ obtained from $\sigma$ by exchanging $j_1$ and $j_2$ in the bottom line. This procedure is reversible. And we have
 $\fix(\sigma)=\fix(\tau)$ and $\inv(\sigma)+\inv(\tau)\equiv 1 \pmod 2$. Then $(-1)^{\inv(\sigma)}+(-1)^{\inv(\tau)}=0$. Thus, we can detete  
 the pair $\{\sigma, \tau\}$ from the symmetry group $\Sym_k$. The value of  the $t$-Hankel determinant $H_{k}(\bf r,t)$ defined by (\ref{equ:tret}) doesn't change. 

After deleting all the permutations such that $\fix(\sigma)\geq 4$,
all remaining permutaions have at most $3$ fixed points.
Thus, the $t$-Hankel determinant $H_{k}(\bf r,t)$ is a polynomial in $t$ of degree less than or equal to 3.
\end{proof}

{\bf Acknowledgements}. The second author should like to thank Zhi-Ying Wen,
who invited him to Tsinghua University for a fruitful two-month visit.
Part of this research was done during that period. 

\vfill\eject

\bibliographystyle{plain}

\bibliography{\jobname}

\end{document}